\documentclass[12pt]{amsart}
\usepackage{amsmath}
\usepackage{amssymb, amsfonts, amsthm}
\usepackage{amssymb}
\usepackage{ifthen}
\usepackage{graphicx}
\usepackage{float}
\usepackage{caption}
\usepackage{subcaption}
\usepackage{cite}
\usepackage{amsfonts}
\usepackage{amscd}
\usepackage{amsxtra}
\usepackage{color}
\numberwithin{equation}{section}

\newtheorem{theorem}{Theorem}[section]
\newtheorem{lemma}[theorem]{Lemma}

\theoremstyle{definition}

\newtheorem{example}[theorem]{Example}

\newtheorem{remark}[equation]{Remark}

\begin{document}
\title[A generalization of the Bohr-Rogosinski sum]
{A generalization of the Bohr-Rogosinski sum}

\author[Shankey Kumar]{Shankey Kumar}
\address{Shankey Kumar, Department of Mathematics,
	Indian Institute of Technology Indore,
	Indore 453552, India}
\email{shankeygarg93@gmail.com}

%\author[S. Ponnusamy]{Saminathan Ponnusamy}
%\address{Saminathan Ponnusamy, Department of Mathematics,
%	Indian Institute of Technology Madras,
%	Chennai 600036, India}
%\email{samy@iitm.ac.in}

\author[S. K. Sahoo]{Swadesh Kumar Sahoo$^*$}
\address{Swadesh Kumar Sahoo, Department of Mathematics,
	Indian Institute of Technology Indore,
	Indore 453552, India}
\email{swadesh.sahoo@iiti.ac.in}

\keywords{Bounded analytic functions, Univalent functions, Bohr radius, Rogosinski radius, Bohr-Rogosinski sum, Subordination.\\
$^*$ The corresponding author}

\begin{abstract}
In this paper, we investigate the Bohr-Rogosinski sum and the classical Bohr sum for analytic functions defined on the unit disk in a general setting. In addition, we discuss a generalization of  the Bohr-Rogosinski sum for a class of analytic functions subordinate to the univalent functions in the unit disk. Several well-known results are observed from the consequences of our main results. 
\end{abstract}

\maketitle

\section{\bf Introduction}
Let $\mathcal{H}$ be the class of all analytic functions defined on the unit disk $\mathbb{D}:=\{z\in \mathbb{C}:|z|<1 \}$. The subclass 
$\mathcal{B}=\{f\in\mathcal{H}: |f(z)|\leq 1\}$,
of $\mathcal{H}$, is our main consideration in this paper. 

The Bohr radius first introduced by Bohr \cite{Bohr14}, who gives that if $f(z)=\sum_{n=0}^{\infty} a_n z^n\in \mathcal{B}$ then
\begin{equation}\label{6eq0.1}
\sum_{n=0}^{\infty} |a_n| r^n\leq 1
\end{equation}
in $|z|=r\leq 1/3$. The constant $1/3$ is known as the {\em Bohr radius}.  Bohr originally obtained the inequality \eqref{6eq0.1} for $r\leq 1/6$ and in the same paper the constant $1/6$ was improved to $1/3$. As pointed out in \cite{Bohr14}, the proof of the sharp radius $1/3$ was suggested by Wiener et al. In the literature, several versions of Bohr's theorem have been studied. One of the versions of Bohr's theorem has been obtained by introducing the term $|a_0|^p$, $0<p\leq 2$, instead of $|a_0|$, with the corresponding radius $p/(2+p)$, see \cite{LP2020,PVW2019}. 

The notion of the Bohr radius was generalized in \cite{Abu,Abu2,Aizenberg07} to include mappings from $\mathbb{D}$ to some other domains in $\mathbb{C}$. Moreover, the Bohr phenomenon for shifted disks and simply connected domains are 
dealt in \cite{AAH21,EPR21,FR10}. 
The Bohr phenomenon for the class of subordinations and the class of
quasi-subordinations are discussed in \cite{BD18} and \cite{AKP19},
respectively.
In \cite{KSS2017}, Kayumov et al. studied the Bohr radius for locally univalent planar harmonic mappings. Various improved forms of the classical Bohr's inequality were investigated by Kayumov and Ponnusamy in \cite{Kayumov18,Kayponn18}. In this sequence, Evdoridis et al. \cite{EPR19} have discussed several improved versions of the Bohr inequality for harmonic mappings. Bohr type inequalities for certain integral operators have been obtained in \cite{Kayponn19,KS2021}. 
To find certain recent results, we refer to \cite{Abu4,AAL20,BD19,Kaypon18,LLP2020,LPW2020,LP2019}
and the references therein. The recent survey article \cite{Abu-M16} and references therein may be good sources for this topic. 

Recently, Kayumov et. al. \cite{Kayponn20} studied the general form of the Bohr sum, which is described as follows:
let $\{\phi_k(r)\}_{k=0}^{\infty}$ be a sequence of non-negative continuous functions in $[0,1)$ such that the series 
$$
\phi_0(r)+\sum_{k=1}^{\infty}\phi_k(r)
$$
converges locally uniformly for $r\in[0,1)$.

\noindent
{\bf Theorem A \cite{Kayponn20}.}
{\em Let $f(z)=\sum_{n=0}^{\infty} a_n z^n \in \mathcal{B}$ and $p\in(0,2]$. If
$$
\phi_0(r)>\frac{2}{p}\sum_{k=1}^{\infty}\phi_k(r)
\quad \mbox{ for $r\in [0,R)$},
$$
where $R$ is the minimal positive root of the equation 
$$
\phi_0(x)=\frac{2}{p}\sum_{k=1}^{\infty}\phi_k(x),
$$
then the following sharp inequality holds:
$$
|a_0|^p \phi_0(r)+\sum_{k=1}^{\infty}|a_k| \phi_k(r)\leq \phi_0(r), \quad \mbox{ for all $r\leq R$}.
$$
In the case when 
$$
\phi_0(x)<\frac{2}{p}\sum_{k=1}^{\infty}\phi_k(x)
$$
in some interval $(R,R+\epsilon)$, the number $R$ cannot be improved. If the functions $\phi_k(x)$ ($k\geq 0$) are smooth functions then the last condition is equivalent to the inequality 
$$
\phi_0^{'}(R)<\frac{2}{p}\sum_{k=1}^{\infty}\phi_k^{'}(R).
$$}

Similar to the Bohr radius, there is a concept of the Rogosinski radius. In \cite{R23}, the Rogosinski radius is defined as follows: if $f(z)=\sum_{n=0}^{\infty} a_n z^n \in \mathcal{B}$ then $|S_M(z)|=|\sum_{n=0}^{M-1} a_n z^n|<1$ for $|z|<1/2$, where $1/2$ is the best possible quantity (see also \cite{LG86,S25}). In \cite{Kayponn21}, Kayumov and Ponnusamy studied the sum 
\begin{equation}\label{6eq0.2}
R_N^f(z):=|f(z)|^p+\sum_{k=N}^{\infty} |a_k| r^k, \ |z|=r, \text{ and } N\in \mathbb{N},
\end{equation}
namely, the {\em Bohr-Rogosinski sum} of $f$ for $p\in\{1,2\}$ and later in \cite{LP2020} this sum has been considered for $p\in (0,2]$. If we choose $N=1$ and $f(0)$ instead of $f(z)$ in the sum then it is easy to see that the Bohr-Rogosinski sum is closely related to the classical Bohr sum. Here, the Bohr-Rogosinski  radius is the largest number $r>0$ such that $R_N^f(z)\leq 1$, known as the {\em Bohr-Rogosinski inequality}, for $|z|\leq r$
and for each $f\in\mathcal{B}$.

For $p\in\{1,2\},$ Liu et al. \cite{Liu18} have considered the Bohr-Rogosinski type sum in the form:
\begin{equation}\label{6eq0.3}
|f(z)|^p+\sum_{k=N}^{\infty}\bigg|\frac{f^k(z)}{k!}\bigg| r^k, \ |z|=r, \text{ and } N\in \mathbb{N},
\end{equation}
which is obtained by replacing $f^{k}(0)$ in \eqref{6eq0.2} with the $k^{th}$ derivative of $f$.
Moreover, the sum \eqref{6eq0.3} has been further generalized in \cite{AlKayPON-20} by replacing $z$ with $z^m$ for a positive integer $m$. 

The main focus of this paper is to consider generalized sums for each of the sums \eqref{6eq0.2} and \eqref{6eq0.3} with an aim to further generalize Theorem A. Concerning these two sums, we present two main results in Section 2 followed by their consequences. Moreover, Section 3 contains a generalization of the Bohr-Rogosinski sum for a class of analytic functions subordinate to univalent functions in the unit disk.

\section{\bf Generalized Bohr-Rogosinski sum for analytic functions}
The following lemma will be instrumental in proving our forthcoming main result.
\begin{lemma}\label{6Lemma1.2}
Let $r\in[0,1)$ and $p\in(0,1]$. Then
$$
Q(x):=1-\bigg(\cfrac{r^m+x}{1+xr^m}\bigg)^p- p\frac{1-r^m}{1+r^m}(1-x)\geq 0, \ \forall \ x\in [0,1) \text{ and } m\in \mathbb{N}.
$$
\end{lemma}
\begin{proof}
Differentiation of $Q(x)$ with respect to $x$ gives
\begin{align*}
Q'(x)&=-p\bigg(\cfrac{r^m+x}{1+xr^m}\bigg)^{p-1}\frac{(1-r^{2m})}{(1+xr^m)^2}+p\frac{1-r^m}{1+r^m}\\
&\leq p\frac{1-r^m}{1+r^m}\Bigg[1-\bigg(\frac{r^m+x}{1+xr^m}\bigg)^{p-1}\Bigg]\leq 0
\end{align*}
for all $x\in [0,1)$ and $r\in[0,1)$. It leads to the fact that $Q$ is a decreasing function of $x$ for $r\in[0,1)$.
So, we have 
$$
Q(x)\geq \lim\limits_{x\rightarrow 1} Q(x)= 0.
$$
This completes the proof.
\end{proof}

Now we are ready to establish our first main result.

\begin{theorem}\label{6theorem1.1}
Let $\{\nu_k(r)\}_{k=0}^{\infty}$ be a sequence of non-negative continuous functions in $[0,1)$ such that the series 
$$
\nu_0(r)+\sum_{k=1}^{\infty}\nu_k(r)
$$
converges locally uniformly with respect to $r\in[0,1)$. 
Let  ${\rm Re}\, f(z) ={\rm Re}\, \sum_{n=0}^{\infty} a_n z^n \leq 1$ and $p\in(0,1]$. If
\begin{equation}\label{6eq1.1}
\nu_0(r)>\frac{2}{p}\frac{(1+r^m)}{(1-r^m)}\sum_{k=1}^{\infty}\nu_k(r)
\end{equation}
then the following sharp inequality holds:
$$
A_f(\nu,p,r,m):=|f(z^m)|^p \nu_0(r)+\sum_{k=1}^{\infty}|a_k| \nu_k(r)\leq \nu_0(r), \mbox{ for all $|z|=r\leq R_1$},
$$
where $R_1$ is the minimal positive root of the equation 
$$
\nu_0(x)=\frac{2}{p}\frac{(1+x^m)}{(1-x^m)}\sum_{k=1}^{\infty}\nu_k(x).
$$
In the case when 
$$
\nu_0(x)<\frac{2}{p}\frac{(1+x^m)}{(1-x^m)}\sum_{k=1}^{\infty}\nu_k(x)
$$
in some interval $(R_1,R_1+\epsilon)$, the number $R_1$ cannot be improved. 
\end{theorem}
\begin{proof} Let $a=|a_0|<1$.
The given condition ${\rm Re}\,f(z)\leq 1$ gives $|a_k|\leq 2(1-a)$ which provides
\begin{align*}
A_f(\nu,p,r,m)&\leq |f(z^m)|^p \nu_0(r)+2(1-a)\sum_{k=1}^{\infty} \nu_k(r)\\
&= \nu_0(r)+2(1-a)\Bigg[\sum_{k=1}^{\infty} \nu_k(r)-\frac{1-|f(z^m)|^p}{2(1-a)}\nu_0(r)\Bigg].
\end{align*}	
Since $f(0)=a_0$ and ${\rm Re}\,f(z)\leq 1$, then we have 
$$
|f(z)|\leq \frac{r+a}{1+ar}, \ |z|=r.
$$
Then by applying Lemma \ref{6Lemma1.2} we have
$$
A_f(\nu,p,r,m)\leq \nu_0(r)+2(1-a)\Bigg[\sum_{k=1}^{\infty} \nu_k(r)-\frac{p}{2}\frac{1-r^m}{1+r^m}\nu_0(r)\Bigg].
$$
The equation \eqref{6eq1.1} provides us
$$
A_f(\nu,p,r,m)\leq \nu_0(r),\quad \mbox{for all $ r\leq R_1$}.
$$
Now, let us prove the sharpness part. Consider a function
$$
g(z)=\frac{a+z}{1+az}=a-(1-a^2)\sum_{k=1}^{\infty}(-1)^k a^{k-1} z^{k},
$$
where $z\in\mathbb{D}$ and $a\in [0,1)$.
Then we obtain
\begin{align*}
A_g(\nu,p,r,m)=& \{g(r^m)\}^p \nu_0(r)+(1-a^2)\sum_{k=1}^{\infty}a^{k-1} \nu_k(r)\\
=& \nu_0(r)+(1-a)\Bigg[2\sum_{k=1}^{\infty} \nu_k(r)-p\frac{1-r^m}{1+r^m}\nu_0(r)\Bigg]\\
&+(1-a)\Bigg[\sum_{k=1}^{\infty}a^{k-1}(1+a) \nu_k(r)-2\sum_{k=1}^{\infty} \nu_k(r)\Bigg]\\
&+\Bigg[p(1-a)\frac{1-r}{1+r}+\{g(r^m)\}^p-1\Bigg]\nu_0(r)\\
=&\nu_0(r)+(1-a)\Bigg[2\sum_{k=1}^{\infty} \nu_k(r)-p\frac{1-r^m}{1+r^m}\nu_0(r)\Bigg]+O((1-a)^2)
\end{align*}
as $a$ tends to $1^-$. Also, we have
$$
2\sum_{k=1}^{\infty} \nu_k(r)>p\frac{1-r^m}{1+r^m}\nu_0(r)
$$
for $r\in(R_1,R_1+\epsilon)$. This concludes the proof.
\end{proof}
%%%%%%%%%%%%%%%%%%%%%%%%%%%%%%%%%%%%%%%%%%%%%%%%%%%%%%%%%%%%
\begin{remark}
Theorem \ref{6theorem1.1}, with the limit $m\to \infty$, is a generalization of Theorem A for $p\in (0,1]$. Incidentally, we came to know through a private communication that Theorem~\ref{6theorem1.1} for the case $m\to \infty$ is recently obtained in \cite{LLP}.
\end{remark}
%%%%%%%%%%%%%%%%%%%%%%%%%%%%%%%%%%%%%%%%%%%%%%%%%%%%%%%%%%%%
Following are the consequences of Theorem \ref{6theorem1.1}.
\begin{example}
For $\nu_0=1$, $\nu_n=r^n$, $n\geq N\in\mathbb{N}$ and $\nu_n=0$, $1\leq n< N$, Theorem \ref{6theorem1.1} gives
$$
|f(z^m)|^p+\sum_{n=N}^{\infty} |a_n| r^n\leq 1, \text{ for } r\leq R_1^{m,N}(p),
$$
where $R_1^{m,N}(p)$ is the positive root of the equation $2x^N(1+x^m)-p(1-x)(1-x^m)=0$. The radius $R_1^{m,N}(p)$ is best possible. The case $|f(z)|\leq 1$ and $p=1$ are already presented in \cite{Kayponn21}. Further, the case $m=1$, $p\in(0,1]$ and $|f(z)|\leq 1$ separately handled in \cite{LP2020}. The roots $R_1^{m,N}(p)$ for $p=1,2$, $N=5,10,15$, and $m=1,2,3,4$ given in Table \ref{table2.3}.   
\end{example}

\begin{table}[H]
	\begin{center}
		\begin{tabular}{|c|c|c|c|c|c|c|}
			\hline
			$m$ & $R_1^{m,5}(1)$ & $R_1^{m,10}(1)$ & $R_1^{m,15}(1)$ & $R_1^{m,5}(2)$ & $R_1^{m,10}(2)$ & $R_1^{m,15}(2)$\\
			\hline
			$1$ &$0.568466$&$0.696983$&$0.760135$&$0.61803$&$0.729092$&$0.78422$\\
			\hline
			$2$ &$0.614046$&$0.727963$&$0.783716$&$0.664727$&$0.759979$&$0.807561$\\
			\hline
			$3$ &$0.638474$&$0.745208$&$0.797015$&$0.690133$&$0.777215$&$0.820717$\\
			\hline
			$4$ &$0.653901$&$0.756719$&$0.80606$&$0.706669$&$0.788828$&$0.8297$\\
			\hline
		\end{tabular}
		\vspace*{0.2cm}
\caption{Computation of $R_1^{m,N}(p)$ for $p=1,2$, $N=5,10,15$, and $m=1,2,3,4$}\label{table2.3}
	\end{center}
\end{table}
%%%%%%%%%%%%%%%%%%%%%%%%%%%%%%%%%%%%%%%%%%%%%%%%%%%%%%%%%%%%
\begin{example}
After letting $\nu_{2n}=r^{2n}$ and $\nu_{2n+1}=0$ ($n\geq 0$) in Theorem $\ref{6theorem1.1}$ we obtain 
$$
|f(z^m)|^p+\sum_{n=1}^{\infty} |a_{2n}| r^{2n}\leq 1, \text{ for } r\leq R_2^m(p),
$$
where $R_2^m(p)$ is the positive root of the equation $2x^2(1+x^m)-p(1-x^2)(1-x^m)=0$. The radius $R_2^m(p)$ is best possible. The choices $p=1$, $m=1$ with $|f(z)|\leq 1$ were studied in \cite{Liu18}. We list a few initial roots  $R_2^m(p)$ for $p=1,2$ in Table~\ref{table1}. 
\end{example}
\begin{table}[H]
\begin{tabular}{|c|c|c|} 
\hline
$m$ & $R_2^m(1)$& $R_2^m(2)$\\ \hline
$1$ & $0.41421$ & $0.5$ \\ \hline 
$2$ & $0.48587$ & $0.57735$\\ \hline
$3$ & $0.52236$ & $0.61803$\\ \hline
$4$ & $0.54369$ & $0.64359$\\ \hline
\end{tabular}
\vspace*{0.2cm}
\caption{Computation of the roots $R_2^m(p)$ for $p=1,2$ and $m=1,2,3,4$.} \label{table1}
\end{table}
%%%%%%%%%%%%%%%%%%%%%%%%%%%%%%%%%%%%%%%%%%%%%%%%%%%%%%%%%%%%
\begin{example}
Let us consider $\nu_0=1$, $\nu_{2n}=0$ and $\nu_{2n-1}=r^{2n-1}$ ($n\geq 1$) in Theorem \ref{6theorem1.1} then we have
$$
|f(z^m)|^p+\sum_{n=1}^{\infty} |a_{2n-1}| r^{2n-1}\leq 1, \text{ for } r\leq R_3^m(p),
$$
where $R_3^m(p)$ is the minimal positive root of the equation $2x(1+x^m)-p(1-x^2)(1-x^m)=0$. The radius $R_3^m(p)$ is best possible. Table~\ref{table2} describes a few initial roots $R_3^m(p)$ for $p=1,2$.
\end{example}
\begin{table}[H]
	\begin{tabular}{|c|c|c|} 
		\hline
		$m$ & $R_3^m(1)$& $R_3^m(2)$\\
		\hline
		$1$ & $0.26795$ & $0.38197$\\ \hline
		$2$ & $0.34601$ & $0.48053$\\ \hline
		$3$ & $0.38197$ & $0.53101$ \\ \hline
		$4$ & $0.399389$& $0.56127$ \\
		\hline
	\end{tabular}
\vspace*{0.2cm}
\caption{The roots $R_3^m(p)$ for $p=1,2$ and $m=1,2,3,4$} \label{table2}
\end{table}
%%%%%%%%%%%%%%%%%%%%%%%%%%%%%%%%%%%%%%%%%%%%%%%%%%%%%%%%%%%%
\begin{example}
The choices $\nu_0=1$, $\nu_n=(n+1)r^n$, $n\geq N\in\mathbb{N}$ and $\nu_n=0$, $1\leq n< N$, in Theorem \ref{6theorem1.1} provide
$$
|f(z^m)|^p+\sum_{n=N}^{\infty}(n+1) |a_n| r^n\leq 1, \text{ for } r\leq R_4^{m,N}(p),
$$
where $R_4^{m,N}(p)$ is the positive root of the equation $2x^N(1+N-Nx)(1+x^m)-p(1-x)^2(1-x^m)=0$. The radius $R_4^{m,N}(p)$ is best possible. The values of $R_4^{m,N}(p)$ have been computed in Table~\ref{table2.6}
for the choices $p=1,2$, $N=5,10,15$, and $m=1,2,3,4$.

\begin{table}[H]
\begin{center}
\begin{tabular}{|c|c|c|c|c|c|c|}
\hline
$m$ & $R_4^{m,5}(1)$ & $R_4^{m,10}(1)$ & $R_4^{m,15}(1)$ & $R_4^{m,5}(2)$ & $R_4^{m,10}(2)$ & $R_4^{m,15}(2)$\\
\hline
$1$ &$0.438303$&$0.581973$&$0.659598$&$0.48227$&$0.612325$&$0.683058$\\
\hline
$2$ &$0.474555$&$0.609805$&$0.681863$&$0.521603$&$0.641053$&$0.705692$\\
\hline
$3$ &$0.491663$&$0.624482$&$0.694054$&$0.54118$&$0.656495$&$0.718209$\\
\hline
$4$ &$0.500617$&$0.633465$&$0.701928$&$0.552248$&$0.66623$&$0.726424$\\
\hline
\end{tabular}
\vspace*{0.2cm}
\caption{Computation of $R_4^{m,N}(p)$ for $p=1,2$, $N=5,10,15$, and $m=1,2,3,4$}\label{table2.6}
\end{center}
\end{table}

\end{example}
%%%%%%%%%%%%%%%%%%%%%%%%%%%%%%%%%%%%%%%%%%%%%%%%%%%%%%%%%%%%
\begin{example}It is easy to calculate 
$$
\sum_{n=N}^{\infty}n r^n=\frac{r^N[N(1-r)+r]}{(1-r)^2}.
$$
Then, for $\nu_0=1$, $\nu_n=nr^n$, $n\geq N\in\mathbb{N}$ and $\nu_n=0$, $1\leq n< N$, Theorem \ref{6theorem1.1} obtains
$$
|f(z^m)|^p+\sum_{n=N}^{\infty}n |a_n| r^n\leq 1, \text{ for } r\leq R_5^{m,N}(p),
$$
where $R_5^{m,N}(p)$ is the positive root of the equation $2x^N[N(1-x)+x](1+x^m)=p(1-x^m)$. The radius $R_5^{m,N}(p)$ is best possible. The roots $R_5^{m,N}(p)$ have been computed in Table~\ref{table2.71}
for the choices $p=1,2$, $N=5,10,15$, and $m=1,2,3,4$.

\begin{table}[H]
	\begin{center}
		\begin{tabular}{|c|c|c|c|c|c|c|}
			\hline
			$m$ & $R_5^{m,5}(1)$ & $R_5^{m,10}(1)$ & $R_5^{m,15}(1)$ & $R_5^{m,5}(2)$ & $R_5^{m,10}(2)$ & $R_5^{m,15}(2)$\\
			\hline
			$1$ &$0.552822$&$0.689323$&$0.75805$&$0.621637$&$0.733723$&$0.79105$\\
			\hline
			$2$ &$0.615996$&$0.732181$&$0.790394$&$0.691438$&$0.77864$&$0.824341$\\
			\hline
			$3$ &$0.652463$&$0.757175$&$0.809302$&$0.732289$&$0.804895$&$0.843792$\\
			\hline
			$4$ &$0.677166$&$0.774533$&$0.82254$&$0.760641$&$0.823255$&$0.857442$\\
			\hline
		\end{tabular}
		\vspace*{0.2cm}
		\caption{The values of $R_5^{m,N}(p)$ for $p=1,2$, $N=5,10,15$, and $m=1,2,3,4$}\label{table2.71}
	\end{center}
\end{table}

Also, we have
$$
\sum_{n=N}^{\infty}n^2 r^n=\frac{r^N[(r+N)^2+r+N^2r^2-2Nr(r+N)]}{(1-r)^3}.
$$
Then, for $\nu_0=1$, $\nu_n=n^2r^n$, $n\geq N\in\mathbb{N}$ and $\nu_n=0$, $1\leq n< N$, in Theorem \ref{6theorem1.1}
$$
|f(z^m)|^p+\sum_{n=N}^{\infty}n^2 |a_n| r^n\leq 1, \text{ for } r\leq R_6^{m,N}(p),
$$
where $R_6^{m,N}(p)$ is the positive root of the equation 
$$
x^N[(x+N)^2+x+N^2x^2-2Nx(x+N)](1+x^m)=p(1-x^m)(1-x)^3.
$$ 
The radius $R_6^{m,N}(p)$ is best possible. Table~\ref{table2.72} describes a few initial roots $R_6^m(p)$ for $p=1,2$ and $N=5,10,15$.
\end{example}

\begin{table}[H]
\begin{center}
\begin{tabular}{|c|c|c|c|c|c|c|}
\hline
$m$ & $R_6^{m,5}(1)$ & $R_6^{m,10}(1)$ & $R_6^{m,15}(1)$ & $R_6^{m,5}(2)$ & $R_6^{m,10}(2)$ & $R_6^{m,15}(2)$\\
\hline
$1$ &$0.384343$&$0.512948$&$0.591745$&$0.423699$&$0.540931$&$0.613822$\\ \hline
$2$ &$0.414554$&$0.537391$&$0.612014$&$0.456968$&$0.566525$&$0.634665$\\ \hline
$3$ &$0.427417$&$0.549435$&$0.622616$&$0.472121$&$0.579505$&$0.645746$\\ \hline
$4$ &$0.433293$&$0.556135$&$0.629021$&$0.479709$&$0.587025$&$0.652601$\\ \hline
\end{tabular}
\vspace*{0.2cm}
\caption{Computation of $R_6^{m,N}(p)$ for $p=1,2$, $N=5,10,15$, and $m=1,2,3,4$}\label{table2.72}
	\end{center}
\end{table}
%%%%%%%%%%%%%%%%%%%%%%%%%%%%%%%%%%%%%%%%%%%%%%%%%%%%%%%%%%%
The next consequence is proved in \cite{Liu18}.
\begin{example}
Theorem \ref{6theorem1.1} gives
$$
|f(z)|+\sum_{k=1}^{\infty} |a_{kn}| r^{kn}\leq 1, \text{ for } r\leq R_7^n,
$$
where $R_7^n$ is the positive root of the equation $2x^n(1+x)-(1-x)(1-x^n)=0$. The radius $R_7^n$ is best possible. For $n=1,2,3,4$, the roots $R_7^n$ are presented in Table~\ref{table3}.
\end{example}
%%%%%%%%%%%%%%%%%%%%%%%%%%%%%%%%%%%%%%%%%%%%%
\begin{table}[H]
	\begin{tabular}{|c|c|c|c|} 
		\hline
		$n$ & $R_7^n$\\
		\hline
		$1$ & $0.23607$ \\ \hline
		$2$ & $0.41421$ \\ \hline
		$3$  & $0.51624$ \\ \hline
		$4$  & $0.58378$\\
		\hline
	\end{tabular}
	\vspace*{0.2cm}
	\caption{The roots $R_7^n$ for $n=1,2,3,4.$} \label{table3}
\end{table}
%%%%%%%%%%%%%%%%%%%%%%%%%%%%%%%%%%%%%%%%%%%%%%%%%%%%%%%%%%%%

%%%%%%%%%%%%%%%%%%%%%%%%%%%%%%%%%%%%%%%%%%%%%%%%%%%%%%%%%%%%
%%%%%%%%%%%%%%%%%%%%%%%%%%%%%%%%%%%%%%%%%%%%%%%%%%%%%%%%%%%%
\begin{lemma}\label{6Lemma1.1}
	Let $r\in[0,1)$ and $p\in(0,2]$. Then
	$$
	P(x):=1-\bigg(\cfrac{r^m+x}{1+xr^m}\bigg)^p- \frac{p}{2}\frac{1-r^m}{1+r^m}(1-x^2)\geq 0
	$$
	for all $x\in [0,1)$ and $m\in \mathbb{N}$.
\end{lemma}
\begin{proof}
	The above inequality can easily be proved by using the following steps: By taking the differentiation of $P(x)$ with respect to $x$, we obtain 
	\begin{align*}
	P'(x)&=-p\bigg(\cfrac{r^m+x}{1+xr^m}\bigg)^{p-1}\frac{(1-r^{2m})}{(1+xr^m)^2}+px\frac{1-r^m}{1+r^m}\\
	&\leq p\frac{1-r^m}{1+r^m}\Bigg[x-\bigg(\frac{r^m+x}{1+xr^m}\bigg)^{p-1}\Bigg]\\
	&\leq p\frac{1-r^m}{1+r^m}\Bigg[x-\frac{r^m+x}{1+xr^m}\Bigg]= p\frac{1-r^m}{1+r^m}\Bigg[\frac{r^m(x^2-1)}{1+xr^m}\Bigg]\leq 0
	\end{align*}
	for all $x\in [0,1)$ and $r\in[0,1)$. Hence $P$ is a decreasing function of $x$ for all $r\in[0,1)$.
	Now, it is easy to calculate 
	$$
	P(x)\geq \lim\limits_{x\rightarrow 1^-} P(x)= 0.
	$$
	This completes the proof of the Lemma.
\end{proof}
%%%%%%%%%%%%%%%%%%%%%%%%%%%%%%%%%%%%%%%%%
The following lemma is due to Ruscheweyh
\cite[Theorem~2]{R85} (see also \cite[Theorem~4.5, p~53]{AW09}).
\begin{lemma}\label{6Lemma1.3}
Let $f \in \mathcal{B}$. Then, for all $k=1,2,\dots$, we have 
$$
\frac{|f^k(z)|}{k!}\leq \frac{(1+|z|)^{k-1}}{(1-|z|^2)^k}(1-|f(z)|^2), \ |z|<1. 
$$
\end{lemma}

Our second main result is presented below.
\begin{theorem}\label{6theorem1.2}
Let $\{\psi_k(r)\}_{k=0}^{\infty}$ be a sequence of non negative continuous functions in $[0,1)$ such that the series 
$$
\psi_0(r)+\sum_{k=1}^{\infty}\frac{(1+r)^{k-1}}{(1-r^2)^k}\psi_k(r)
$$
converges locally uniformly with respect to $r\in[0,1)$.
Let $f(z)=\sum_{n=0}^{\infty} a_n z^n \in \mathcal{B}$ and $p\in(0,2]$. If
\begin{equation}\label{6eq1.3}
\psi_0(r)>\frac{2}{p}\sum_{k=1}^{\infty}\frac{(1+r^m)^{k-1}}{(1-r^{2m})^k}\psi_k(r)
\end{equation}
then the following sharp inequality holds:
$$
B_f(\psi,p,r,m):=|f(z^m)|^p \psi_0(r)+\sum_{k=1}^{\infty}\bigg|\frac{f^k(z^m)}{k!}\bigg| \psi_k(r)\leq \psi_0(r), \mbox{ for all $|z|=r\leq R_2$},
$$
where $R_2$ is the minimal positive root of the equation 
$$
\psi_0(x)=\frac{2}{p}\sum_{k=1}^{\infty}\frac{(1+x^m)^{k-1}}{(1-x^{2m})^k}\psi_k(x).
$$
In the case when 
$$
\psi_0(x)<\frac{2}{p}\sum_{k=1}^{\infty}\frac{(1+x^m)^{k-1}}{(1-x^{2m})^k}\psi_k(x)
$$
in some interval $(R_2,R_2+\epsilon)$, the number $R_2$ cannot be improved. 
\end{theorem}
\begin{proof}
Given that $f\in \mathcal{B}$. Then by using Lemma \ref{6Lemma1.3} we obtain
\begin{align*}
B_f(\psi,p,r,m)&\leq |f(z^m)|^p \psi_0(r)+(1-|f(z^m)|^2)\sum_{k=1}^{\infty}\frac{(1+|z|^m)^{k-1}}{(1-|z|^{2m})^k} \psi_k(r)\\
&=\psi_0(r)+(1-|f(z^m)|^2)\Bigg[\sum_{k=1}^{\infty}\frac{(1+|z|^m)^{k-1}}{(1-|z|^{2m})^k} \psi_k(r)-\frac{(1-|f(z^m)|^p)}{(1-|f(z^m)|^2)}\psi_0(r)\Bigg].
\end{align*}
By using Lemma \ref{6Lemma1.1}, for $r=0$, it provides us 
$$
B_f(\psi,p,r,m)	\leq  \psi_0(r)+(1-|f(z^m)|^2)\Bigg[\sum_{k=1}^{\infty}\frac{(1+|z|^m)^{k-1}}{(1-|z|^{2m})^k} \psi_k(r)-\frac{p}{2}\psi_0(r)\Bigg].
$$
The equation \eqref{6eq1.3} gives
$$
B_f(\psi,p,r,m)\leq \psi_0(r),
$$
for all $r\leq R_2$. This completes the first part of the theorem.

To prove the final part we consider a function
$$
h(z)= \frac{a-z}{1-az}=a-(1-a^2)\sum_{k=1}^{\infty}a^{k-1} z^{k},
$$
where $z\in\mathbb{D}$ and $a\in [0,1)$. Then we obtain
\begin{align*}
B_f(\psi,p,r,m)=& \{h(r^m)\}^p \psi_0(r)+(1-a^2)\sum_{k=1}^{\infty}\frac{a^{k-1}}{(1-ar^m)^{k+1}} \psi_k(r)\\
=& \psi_0(r)+(1-a)\Bigg[\sum_{k=1}^{\infty}\frac{2}{(1-r^m)^{k+1}} \psi_k(r)-p\frac{1+r^m}{1-r^m}\psi_0(r)\Bigg]\\
&+(1-a)\Bigg[\sum_{k=1}^{\infty}\frac{a^{k-1}(1+a)}{(1-ar^m)^{k+1}} \psi_k(r)-\sum_{k=1}^{\infty}\frac{2}{(1-r^m)^{k+1}} \psi_k(r)\Bigg]\\
&+\Bigg[p(1-a)\frac{1+r^m}{1-r^m}+\{h(r^m)\}^p-1\Bigg]\psi_0(r)\\
=&\psi_0(r)+(1-a)\Bigg[\sum_{k=1}^{\infty}\frac{2}{(1-r^m)^{k+1}} \psi_k(r)-p\frac{1+r^m}{1-r^m}\psi_0(r)\Bigg]+O((1-a)^2)
\end{align*}
as $a$ tends to $1^-$. Also, if $a$ is close to $1$ then
$$
\sum_{k=1}^{\infty}\frac{2}{(1-r^m)^{k+1}} \psi_k(r)>p\frac{1+r^m}{1-r^m}\psi_0(r)
$$
for $r\in(R_2,R_2+\epsilon)$. This completes the proof.
\end{proof}
We can obtain several known results as  consequences of the above theorem. They are presented below.
\begin{example}
For $\psi_0=1$, $\psi_n=r^n$, $n\geq N\in\mathbb{N}$ and $\psi_n=0$, $1\leq n< N$, Theorem \ref{6theorem1.2} gives 
$$
|f(z^m)|^p +\sum_{k=N}^{\infty}\bigg|\frac{f^k(z^m)}{k!}\bigg| r^k\leq 1, \quad \mbox{ for all} \ |z|=r\leq R_8^{m,N}(p),
$$
where $R_8^{m,N}(p)$ is the minimal positive root of the equation $2x^N-p(1-x^{2m})(1-x-x^m)=0$. The radius $R_8^{m,N}(p)$ is best possible. Recently, the cases $p=1$ and $N=2$ were considered in \cite{AlKayPON-20}. Also, the situations $m=1$ and $p=1 \text{ or } 2$ were investigated in \cite{Liu18}. For $p=1,2$ and $N=5,10,15$, the values $R_8^{m,N}(p)$ are computed in Table~\ref{table2.8}.
\end{example}
%%%%%%%%%%%%%%%%%%%%%%%%%%%%%%%%%%%%%%%%%%%%%%%%%%%%%%%%%%%
\begin{table}[H]
\begin{center}
\begin{tabular}{|c|c|c|c|c|c|c|}
\hline
$m$ & $R_8^{m,5}(1)$ & $R_8^{m,10}(1)$ & $R_8^{m,15}(1)$ & $R_8^{m,5}(2)$ & $R_8^{m,10}(2)$ & $R_8^{m,15}(2)$\\
\hline
$1$ &$0.470417$&$0.498733$&$0.499959$&$0.482881$&$0.499358$&$0.49998$\\ \hline
$2$ &$0.561279$&$0.610534$&$0.617281$&$0.583333$&$0.614053$&$0.617654$\\ \hline
$3$ &$0.605857$&$0.666331$&$0.679507$&$0.634512$&$0.673433$&$0.680874$\\ \hline
$4$ &$0.632413$&$0.699984$&$0.718457$&$0.666291$&$0.710367$&$0.721294$\\ \hline
\end{tabular}
\vspace*{0.2cm}
\caption{Computation of $R_8^{m,N}(p)$ for $p=1,2$, $N=5,10,15$, and $m=1,2,3,4$}\label{table2.8}
\end{center}
\end{table}
%%%%%%%%%%%%%%%%%%%%%%%%%%%%%%%%%%%%%%%%%%%%%%%%%%%%%%%%%%%
\begin{example}
After letting $\psi_0=1$, $\psi_{2n}=0$ ($1\leq n< N$), $\psi_{2n}=r^{2n}$ ($n\geq N$) and $\psi_{2n-1}=0$ ($n\geq 1$) in Theorem \ref{6theorem1.2} we obtain
$$
|f(z^m)|^p +\sum_{k=N}^{\infty}\bigg|\frac{f^{2k}(z^m)}{(2k)!}\bigg| r^{2k}\leq 1, \quad 
\mbox{for all $\ |z|=r\leq R_9^{m,N}(p)$},
$$
where $R_9^{m,N}(p)$ is the minimal positive root of the equation 
$$
2x^{2N}-p(1+x^m)(1-x^m)^{2(N-1)}[(1-x^m)^2-x^2]=0.
$$ The radius $R_9^{m,N}(p)$ is best possible. For $p=1,2$ and $N=5,10,15$, the roots $R_9^{m,N}(p)$ are presented in Table~\ref{table2.9}.
\end{example}
%%%%%%%%%%%%%%%%%%%%%%%%%%%%%%%%%%%%%%%%%%%%%%%%%%%%%%%%%%%
\begin{table}[H]
\begin{center}
\begin{tabular}{|c|c|c|c|c|c|c|}
\hline
$m$ & $R_9^{m,5}(1)$ & $R_9^{m,10}(1)$ & $R_9^{m,15}(1)$ & $R_9^{m,5}(2)$ & $R_9^{m,10}(2)$ & $R_9^{m,15}(2)$\\
\hline
$1$ &$0.459924$&$0.470621$&$0.481132$&$0.470621$&$0.480648$&$0.485091$\\ \hline
$2$ &$0.570642$&$0.588913$&$0.596301$&$0.58335$&$0.595553$&$0.600837$\\ \hline
$3$ &$0.631077$&$0.651295$&$0.659307$&$0.644949$&$0.658391$&$0.664114$\\ \hline
$4$ &$0.670628$&$0.692296$&$0.700735$&$0.68537$&$0.699696$&$0.705709$\\ \hline
\end{tabular}
\vspace*{0.2cm}
\caption{The values $R_9^{m,N}(p)$ for $p=1,2$, $N=5,10,15$, and $m=1,2,3,4$}\label{table2.9}
\end{center}
\end{table}
\begin{example}
The choices $\psi_0=1$, $\psi_{2n-1}=0$ ($1\leq n< N$), $\psi_{2n-1}=r^{2n-1}$ ($n\geq N$) and $\psi_{2n}=0$ ($n\geq 1$) in Theorem \ref{6theorem1.2} provide
$$
|f(z^m)|^p+\sum_{k=N}^{\infty}\bigg|\frac{f^{2k-1}(z^m)}{(2k-1)!}\bigg| r^{2k-1}\leq 1, \mbox{ for all $|z|=r\leq R_{10}^{m,N}(p)$},
$$
where $R_{10}^{m,N}(p)$ is the minimal positive root of the equation 
$$2x^{2N-1}-p(1+x^m)(1-x^m)^{2N-3}[(1-x^m)^2-x^2]=0.
$$ 
The radius $R_{10}^{m,N}(p)$ is best possible. Table~\ref{table2.10} describes a few initial roots $R_{10}^{m,N}(p)$ for $p=1,2$ and $N=5,10,15$.
\end{example}
%%%%%%%%%%%%%%%%%%%%%%%%%%%%%%%%%%%%%%%%%%%%%%%%%%%%%%%%%%%
\begin{table}[H]
\begin{center}
\begin{tabular}{|c|c|c|c|c|c|c|}
\hline
$m$ & $R_{10}^{m,5}(1)$ & $R_{10}^{m,10}(1)$ & $R_{10}^{m,15}(1)$ & $R_{10}^{m,5}(2)$ & $R_{10}^{m,10}(2)$ & $R_{10}^{m,15}(2)$\\
\hline
$1$ &$0.457053$&$0.474009$&$0.480671$&$0.468802$&$0,480015$&$0.484755$\\ \hline
$2$ &$0.567068$&$0.587828$&$0.595758$&$0.581095$&$0.594794$&$0.600441$\\ \hline
$3$ &$0.627057$&$0.65011$&$0.658721$&$0.642428$&$0.657565$&$0.663687$\\ \hline
$4$ &$0.666256$&$0.691041$&$0.7001205$&$0.682652$&$0.698824$&$0.705261$\\ \hline
\end{tabular}
\vspace*{0.2cm}
\caption{Computation of $R_{10}^{m,N}(p)$ for $p=1,2$, $N=5,10,15$, and $m=1,2,3,4$}\label{table2.10}
\end{center}
\end{table}
%%%%%%%%%%%%%%%%%%%%%%%%%%%%%%%%%%%%%%%%%%%%%%%%%%%%%%%%%%%%%%%%%%%%%%%%%%%%%%%%%%%%%%%%%%%%%%%%%%%%%%%%
\section{\bf Generalized Bohr-Rogosinski sum for a class of subordinations}
Suppose $h\in \mathcal{B}$ of the form $h(z)=\sum_{k=0}^{\infty} c_k z^k$. Then the Bohr-Rogosinski inequality for the function $h$ is 
$$
|h(z)|+\sum_{k=1}^{\infty} c_k r^k \leq 1, \ |z|=r.
$$
We can rewrite the above inequality as 
%\begin{equation}%\label{6eq3.1}
$$
\sum_{k=1}^{\infty} c_k r^k \leq 1-|h(z)|
={\rm dist}\,(h(z),\partial \mathbb{D}), \ |z|=r.
$$
%\end{equation}
%Here, it is easy to observe that the quantity $1-|h(z)|={\rm dist}\,(h(z),\partial \mathbb{D})$, the Euclidean distance from $h(z)$ to the unit circle
%$\partial \mathbb{D}$. Then the inequality \eqref{6eq3.1} becomes
%$$
%\sum_{k=1}^{\infty} c_k r^k \leq \mbox{dist}\,(h(z),\partial \mathbb{D}), \ |z|=r.
%$$
The above inequality can be studied for the generalized class of functions $f$ which is analytic in $\mathbb{D}$ and  $f(\mathbb{D})=\Omega$, for a given domain $\Omega$. To go further,  we recall the definition of subordination here. 

Suppose $f$ and $g$ are analytic functions in $\mathbb{D}$. We say that $g$ is subordinate to $f$, or $g\prec f$, if there is an analytic function 
$w:\mathbb{D}\rightarrow \mathbb{D}$ with $w(0)=0$ such that $g=f\circ w$. 
Note that if $f$ is univalent then the condition $g\prec f$ is equivalent 
to the conditions $f(0)=g(0)$ and $\{g(z):|z|<r\}\subset \{f(z):|z|<r\}$, $r\leq 1$. 
To know more about subordination, reader can refer to \cite{Duren83,MM-Book,Pom-Book}.

In \cite{Abu}, Abu-Muhanna studied the Bohr sum for the class $S(f):=\{g:g\prec f\}$, where $f$ is a univalent function and $f(\mathbb{D})=\Omega$. Recently, Kayumov et al. \cite{Kayponn21} find a radius $R$ for which a generalization of the Bohr-Rogosinski inequality, for the function $g(z)=\sum_{k=0}^{\infty} b_k z^k\in S(f)$,
$$
|g(z^m)|+\sum_{k=N}^{\infty} b_k r^k \leq |f(0)|+ \mbox{dist}\,(f(0),\partial \Omega), \ |z|=r \leq R \mbox{ and } m,N \in \mathbb{N}
$$
holds. More about this result will be discussed in the list of consequences of the following result:   
\begin{theorem}\label{6theorem3.1}
Let $\{\varphi_k(r)\}_{k=1}^{\infty}$ be a sequence of non-negative continuous functions in $[0,1)$ such that the series 
$$
\sum_{k=1}^{\infty}k\varphi_k(r)
$$
converges locally uniformly with respect to $r\in[0,1)$. 
Assume that $f$ and $g$ are analytic in $\mathbb{D}$ such that $f$ is univalent in $\mathbb{D}$ and $g(z)=\sum_{k=0}^{\infty} b_k z^k\in S(f)$. If
\begin{equation}\label{6eq3.2}
\sum_{k=1}^{\infty}k\varphi_k(r)+\frac{r^m}{(1-r^m)^2}<\frac{1}{4}, \ m\in \mathbb{N}
\end{equation}
then the following sharp inequality holds:
$$
C_f(\varphi,r,m):=|g(z^m)|+\sum_{k=1}^{\infty}|b_k|\varphi_k(r)\leq |f(0)|+ \mbox{dist}\,(f(0), \partial \Omega), \mbox{ for all } |z|=r\leq R_3,
$$
where $R_3$ is the minimal positive root of the equation 
$$
\sum_{k=1}^{\infty}k\varphi_k(x)+\frac{x^m}{(1-x^m)^2}=\frac{1}{4}.
$$
In the case when 
$$
\sum_{k=1}^{\infty}k\varphi_k(x)+\frac{x^m}{(1-x^m)^2}>\frac{1}{4}
$$
in some interval $(R_3,R_3+\epsilon)$, the number $R_3$ cannot be improved. 
\end{theorem}
\begin{proof} The univalent condition on function $f$ provides us the well-known inequality 
\begin{equation}\label{6eq3.3}
\frac{1}{4} |f'(z)|(1-|z|^2)\leq \mbox{dist}\,(f(z), \partial \Omega)\leq |f'(z)|(1-|z|^2), \mbox{ for all } z\in \mathbb{D}, 
\end{equation}
see, for instance \cite{Abu,Duren83,Pom-Book2}. Also, the assumption $g(z)=\sum_{k=0}^{\infty} b_k z^k \prec f(z)$ gives $|b_k|\leq k|f'(0)|$, for all $k\in \mathbb{N}$. Then, by using the inequality \eqref{6eq3.3}, we have $|b_k|\leq 4k\mbox{dist}\,(f(0), \partial \Omega)$. It follows that 
$$
C_f(\varphi,r,m)\leq |g(z^m)|+4\mbox{dist}\,(f(0), \partial \Omega)\sum_{k=1}^{\infty}k\varphi_k(r).
$$	
The condition $g\prec f$ and the growth theorem \cite[Theorem 2.6]{Duren83} lead to the fact that 
$$
|g(z)-g(0)|\leq |f'(0)|\frac{r}{(1-r)^2}, \ |z|=r.
$$
Moreover, the inequality \eqref{6eq3.3} gives
$$
|g(z)|\leq |f(0)|+ \mbox{dist}\,(f(0), \partial \Omega) \frac{4r}{(1-r)^2}, \ |z|=r.
$$
Then we obtain
\begin{align*}
C_f(\varphi,r,m)&\leq |f(0)|+ \mbox{dist}\,(f(0), \partial \Omega) \frac{4r^m}{(1-r^m)^2}+4\mbox{dist}\,(f(0), \partial \Omega)\sum_{k=1}^{\infty}k\varphi_k(r)\\
&=|f(0)|+ \mbox{dist}\,(f(0), \partial \Omega)+4\mbox{dist}\,(f(0), \partial \Omega) \Bigg[\frac{r^m}{(1-r^m)^2}+\sum_{k=1}^{\infty}k\varphi_k(r)-\frac{1}{4}\Bigg].
\end{align*}
By using the inequality \eqref{6eq3.2} we have
$$
C_f(\varphi,r,m)\leq |f(0)|+ \mbox{dist}\,(f(0), \partial \Omega), \mbox{ for all } r\leq R_3.
$$
The choice of the function 
$$
f(z)=\frac{z}{(1-z)^2}, \ z\in \mathbb{D}
$$
gives {\rm dist}\,($f(0),  \partial \Omega$)= $1/4$. Also, we have 
$$
f(r^m)+\sum_{k=1}^{\infty}k\varphi_k(r)= \frac{r^m}{(1-r^m)^2}+\sum_{k=1}^{\infty}k\varphi_k(r)>\frac{1}{4}=|f(0)|+ \mbox{dist}\,(f(0), \partial \Omega),
$$
for $r\in(R_3,R_3+\epsilon)$. This gives that we can not improve $R_3$.
\end{proof}
\begin{remark}
For $m\to \infty$, Theorem \ref{6theorem3.1} gives that: if 
$$
\sum_{k=1}^{\infty}k\varphi_k(r)<\frac{1}{4}
$$
then the following sharp inequality holds:
$$
\sum_{k=1}^{\infty}|b_k|\varphi_k(r)\leq \mbox{dist}\,(f(0), \partial \Omega), \mbox{ for all } |z|=r\leq R_4,
$$
where $R_4$ is the minimal positive root of the equation 
$$
\sum_{k=1}^{\infty}k\varphi_k(x)=\frac{1}{4}.
$$
In the case when 
$$
\sum_{k=1}^{\infty}k\varphi_k(x)>\frac{1}{4}
$$
in some interval $(R_4,R_4+\epsilon)$, the number $R_4$ cannot be improved. Now, the particular choices of the functions $\varphi_k(r)=r^k$ give a result of \cite{Abu}.	
\end{remark}
%%%%%%%%%%%%%%%%%%%%%%%%%%%%%%%%%%%%%%%%%%%%%%%%%%%%%%%%
\begin{example}
For $N\in \mathbb{N}$, the choices $\varphi_k(r)=0$, for $1\leq k<N$, and $\varphi_k(r)=r^k$, for $k\geq N$, in Theorem \ref{6theorem3.1} give
$$
|g(z^m)|+\sum_{k=N}^{\infty}|b_k|r^k\leq |f(0)|+ \mbox{dist}\,(f(0), \partial \Omega), \mbox{ for all } |z|=r\leq R_{11}^{m,N},
$$
where $g\in S(f)$ and $R_{11}^{m,N}$ is the positive root of the equation
$$
4x^m-(1-x^m)^2+4x^N[N(1-x)+x]\bigg(\frac{1-x^m}{1-x}\bigg)^2=0.
$$
The radius $R_{11}^{m,N}$ is best possible. This result is proved in \cite{Kayponn21}. For $m=1,2,3,4$ and $N=5,10,15$, the roots $R_{11}^{m,N}$ are presented in Table~\ref{table3.11}.
\end{example}
%%%%%%%%%%%%%%%%%%%%%%%%%%%%%%%%%%%%%%%%%%%%%%%%%%%%%%%%%%%
\begin{table}[H]
\begin{center}
\begin{tabular}{|c|c|c|c|}
\hline
$m$ & $R_{11}^{m,5}$ & $R_{11}^{m,10}$ & $R_{11}^{m,15}$ \\
\hline
$1$ &$0.171125$&$0.171573$&$0.171573$\\ \hline
$2$ &$0.372068$&$0.412677$&$0.414185$\\ \hline
$3$ &$0.432697$&$0.531244$&$0.553009$\\ \hline
$4$ &$0.453269$&$0.576975$&$0.624641$\\ \hline
\end{tabular}
\vspace*{0.2cm}
\caption{Computation of $R_{11}^{m,N}$ for $N=5,10,15$ and $m=1,2,3,4$}\label{table3.11}
\end{center}
\end{table}
%%%%%%%%%%%%%%%%%%%%%%%%%%%%%%%%%%%%%%%%%%%%%%%%%%%%%%%%%%
\begin{example}
For $k\in \mathbb{N}$, the settings $\varphi_{2k-1}(r)=0$ and $\varphi_{2k}(r)=r^{2k}$ in Theorem \ref{6theorem3.1} give
$$
|g(z^m)|+\sum_{k=1}^{\infty}|b_{2k}|r^{2k}\leq |f(0)|+ \mbox{dist}\,(f(0), \partial \Omega), \mbox{ for all } |z|=r\leq R_{12}^{m},
$$
where $g\in S(f)$ and $R_{12}^{m}$ is the minimal positive root of the equation
\begin{equation}\label{6eq3.4}
\frac{2x^2}{(1-x^2)^2} + \frac{x^m}{(1-x^m)^2}=\frac{1}{4}.
\end{equation}
The radius $R_{12}^{m}$ is best possible. Table~\ref{table4} listed the values of $R_{12}^{m}$ for certain choices of $m$.
\end{example}
%%%%%%%%%%%%%%%%%%%%%%%%%%%%%%%%%%%%%%%%%%%%%%%%%%%%%%%%%%%%%%
\begin{table}[H]
	\begin{tabular}{|c|c|c|c|} 
		\hline
		$m$ & $R_{12}^{m}$ \\
		\hline
		$1$ & $0.14813$ \\ \hline
		$2$ & $0.26795$ \\ \hline
		$3$ & $0.30200$ \\ \hline
		$4$ & $0.31270$ \\
		\hline
	\end{tabular}
	\vspace*{0.2cm}
	\caption{Values of $R_{12}^{m}$ for $m=1,2,3,4$} \label{table4}
\end{table}
%%%%%%%%%%%%%%%%%%%%%%%%%%%%%%%%%%%%%%%%%%%%%%%%%%%%%%%%%%%%%%%
%%%%%%%%%%%%%%%%%%%%%%%%%%%%%%%%%%%%%%%%%%%%%%%%%%%%%%%%%%
\begin{example}
Let $\varphi_{2k}(r)=0$ and $\varphi_{2k-1}(r)=r^{2k-1}$, $k\in\mathbb{N}$. Then Theorem \ref{6theorem3.1} gives
$$
|g(z^m)|+\sum_{k=1}^{\infty}|b_{2k-1}|r^{2k-1}\leq |f(0)|+ \mbox{dist}\,(f(0), \partial \Omega), \mbox{ for all } |z|=r\leq R_{13}^{m},
$$
where $g\in S(f)$ and $R_{13}^{m}$ is the minimal positive root of the equation
\begin{equation}\label{6eq3.5}
\frac{x^3(3-x^2)}{(1-x^2)^2} + \frac{x^m}{(1-x^m)^2}=\frac{1}{4}.
\end{equation}
The radius $R_{13}^{m}$ is best possible. For $m=1,2,3,4$, the values of $R_{13}^{m}$ are presented in Table~\ref{table5}.
\end{example}
%%%%%%%%%%%%%%%%%%%%%%%%%%%%%%%%%%%%%%%%%%%%%%%%%%%%%%%%%%
\begin{table}[H]
	\begin{tabular}{|c|c|c|c|} 
		\hline
		$m$ & $R_{13}^{m}$ \\
		\hline
		$1$ & $0.164662$ \\ \hline
		$2$ & $0.322256$ \\ \hline
		$3$ & $0.369627$ \\ \hline
		$4$ & $0.386157$ \\
		\hline
	\end{tabular}
	\vspace*{0.2cm}
	\caption{The values of $R_{13}^{m}$ for $m=1,2,3,4.$} \label{table5}
\end{table}
%%%%%%%%%%%%%%%%%%%%%%%%%%%%%%%%%%%%%%%%%%%%%%%%%%%%%%%%%%%%%%%
In Theorem \ref{6theorem3.1}, if we add convexity condition on the  function $f$ then we have the following:
\begin{theorem}\label{6theorem3.2}
Let $\{\lambda_k(r)\}_{k=1}^{\infty}$ be a sequence of non-negative continuous functions in $[0,1)$ such that the series 
$$
\sum_{k=1}^{\infty}\lambda_k(r)
$$
converges locally uniformly with respect to $r\in[0,1)$. 
Assume that $f$ and $g$ are analytic in $\mathbb{D}$ such that $f$ is  convex univalent in $\mathbb{D}$ and $g\in S(f)$. If
$$
\sum_{k=1}^{\infty}\lambda_k(r)+\frac{r^m}{1-r^m}<\frac{1}{2}
$$
then the following sharp inequality holds:
$$
D_f(\lambda,r,m):=|g(z^m)|+\sum_{k=1}^{\infty}|b_k|\lambda_k(r)\leq |f(0)|+ {\rm dist}\,(f(0), \partial \Omega), \mbox{ for all } |z|=r\leq R_5,
$$
where $R_5$ is the minimal positive root of the equation 
$$
\sum_{k=1}^{\infty}\lambda_k(x)+\frac{x^m}{1-x^m}=\frac{1}{2}.
$$
In the case when 
$$
\sum_{k=1}^{\infty}\lambda_k(x)+\frac{x^m}{1-x^m}>\frac{1}{2}
$$
in some interval $(R_5,R_5+\epsilon)$, the number $R_5$ cannot be improved. 
\end{theorem}
\begin{proof}
Given that $f$ is convex univalent function which gives us the well-known inequality 
$$
\frac{1}{2} |f'(z)|(1-|z|^2)\leq \mbox{dist}\,(f(z), \partial \Omega)\leq |f'(z)|(1-|z|^2), \mbox{ for all } z\in \mathbb{D}, 
$$
see, for instance \cite{Abu,Duren83}. Further, the assumption $g(z)=\sum_{k=0}^{\infty} b_k z^k \prec f(z)$ gives $|b_k|\leq |f'(0)|$, for all $k\in \mathbb{N}$, and
$$
|g(z)-g(0)|\leq |f'(0)|\frac{r}{1-r}.
$$
Rest of the proof follows similar to Theorem \ref{6theorem3.1}. The convex function $f(z)=z/(1-z)$ have \mbox{dist}\,$(f(z), \partial \Omega)=1/2$, $z\in\mathbb{D}$. Moreover, this function gives sharpness of the result.   
\end{proof}
%%%%%%%%%%%%%%%%%%%%%%%%%%%%%%%%%%%%%%%%%%%%%%%%%%%%%%%%%%%%
\begin{remark}
Let $m$ tend to $\infty$ in Theorem \ref{6theorem3.2}. Then one can easily observe that: if
$$
\sum_{k=1}^{\infty}\lambda_k(r)<\frac{1}{2}
$$
then the following sharp inequality holds:
$$
\sum_{k=1}^{\infty}|b_k|\lambda_k(r)\leq \mbox{dist}\,(f(0), \partial \Omega), \mbox{ for all } |z|=r\leq R_6,
$$
where $R_6$ is the minimal positive root of the equation 
$$
\sum_{k=1}^{\infty}\lambda_k(x)=\frac{1}{2}.
$$
In the case when 
$$
\sum_{k=1}^{\infty}\lambda_k(x)>\frac{1}{2}
$$
in some interval $(R_6,R_6+\epsilon)$, the number $R_6$ cannot be improved. Further, the choice of the functions $\varphi_k(r)=r^k$ gives a known result, which is proved in \cite{Abu}.	
\end{remark}
\begin{example}
If we consider $\lambda_k(r)=0$, for $1\leq k<N$, and $\lambda_k(r)=r^k$, for $k\geq N$, in Theorem \ref{6theorem3.2}. Then we have
$$
|g(z^m)|+\sum_{k=N}^{\infty}|b_k|r^k\leq |f(0)|+ \mbox{dist}\,(f(0), \partial \Omega), \mbox{ for all } |z|=r\leq R_{14}^{m,N},
$$
where $g\in S(f)$ and $R_{14}^{m,N}$ is the positive root of the equation
$$
3x^m-1+2x^N\bigg(\frac{1-x^m}{1-x}\bigg)=0.
$$
The radius $R_{14}^{m,N}$ is best possible. This result is studied in \cite{Kayponn21}. The roots $R_{14}^{m,N}$, for $m=1,2,3,4$ and $N=5,10,15$, are given in Table~\ref{table3.14}.
\end{example}
%%%%%%%%%%%%%%%%%%%%%%%%%%%%%%%%%%%%%%%%%%%%%%%%%%%%%%%%%%%
\begin{table}[H]
\begin{center}
\begin{tabular}{|c|c|c|c|}
\hline
$m$ & $R_{14}^{m,5}$ & $R_{14}^{m,10}$ & $R_{14}^{m,15}$ \\
\hline
$1$ &$0.330697$&$0.333322$&$0.333333$\\ \hline
$2$ &$0.536482$&$0.573823$&$0.577111$\\ \hline
$3$ &$0.607547$&$0.673834$&$0.689549$\\ \hline
$4$ &$0.640031$&$0.719763$&$0.746595$\\ \hline
\end{tabular}
\vspace*{0.2cm}
\caption{Values of $R_{14}^{m,N}$ for $N=5,10,15$, and $m=1,2,3,4$}\label{table3.14}
\end{center}
\end{table}
%%%%%%%%%%%%%%%%%%%%%%%%%%%%%%%%%%%%%%%%%%%%%%%%%%%%%%%%%%
%%%%%%%%%%%%%%%%%%%%%%%%%%%%%%%%%%%%%%%%%%%%%%%%%%%%%%%%%%
\begin{example}
Let $\lambda_{2k-1}(r)=0$ and $\lambda_{2k}(r)=r^{2k}$, $k\in\mathbb{N}$ in Theorem \ref{6theorem3.2}. Then we obtain
$$
|g(z^m)|+\sum_{k=1}^{\infty}|b_{2k}|r^{2k}\leq |f(0)|+ \mbox{dist}\,(f(0), \partial \Omega), \mbox{ for all } |z|=r\leq R_{15}^{m},
$$
where $g\in S(f)$ and $R_{15}^{m}$ is the positive root of the equation
\begin{equation}\label{6eq3.6}
\frac{x^2}{1-x^2} + \frac{x^m}{1-x^m}=\frac{1}{2}.
\end{equation}
The radius $R_{15}^{m}$ is best possible.
Table~\ref{table6} computes the values $R_{15}^{m}$ for $m=1,2,3,4.$
\end{example}
%%%%%%%%%%%%%%%%%%%%%%%%%%%%%%%%%%%%%%%%%%%%%%%%%%%%%%%%%%%%%%
\begin{table}[H]
	\begin{tabular}{|c|c|c|c|} 
		\hline
		$m$ & $R_{15}^{m}$ \\
		\hline
		$1$ & $0.28990$ \\ \hline
		$2$ & $0.44721$ \\ \hline
		$3$ & $0.50845$ \\ \hline
		$4$ & $0.53842$ \\
		\hline
	\end{tabular}
	\vspace*{0.2cm}
	\caption{Computation of $R_{15}^{m}$ for $m=1,2,3,4.$} \label{table6}
\end{table}
%%%%%%%%%%%%%%%%%%%%%%%%%%%%%%%%%%%%%%%%%%%%%%%%%%%%%%%%%%%%%%%
%%%%%%%%%%%%%%%%%%%%%%%%%%%%%%%%%%%%%%%%%%%%%%%%%%%%%%%%%%
\begin{example}
For $k\in \mathbb{N}$, the settings $\lambda_{2k}(r)=0$ and $\lambda_{2k-1}(r)=r^{2k-1}$ in Theorem \ref{6theorem3.2} give
$$
|g(z^m)|+\sum_{k=1}^{\infty}|b_{2k-1}|r^{2k-1}\leq |f(0)|+ \mbox{dist}\,(f(0), \partial \Omega), \mbox{ for all } |z|=r\leq R_{16}^{m},
$$
where $g\in S(f)$ and $R_{16}^{m}$ is the positive root of the equation
\begin{equation}\label{6eq3.7}
\frac{x}{1-x^2} + \frac{x^m}{1-x^m}=\frac{1}{2}.
\end{equation}
The radius $R_{16}^{m}$ is best possible. Computation of $R_{16}^{m}$, for $m=1,2,3,4$, is given in Table~\ref{table7}.
\end{example}
%%%%%%%%%%%%%%%%%%%%%%%%%%%%%%%%%%%%%%%%%%%%%%%%%%%%%%%%%%%%%%
\begin{table}[H]
	\begin{tabular}{|c|c|c|c|} 
		\hline
		$m$ & $R_{16}^{m}$ \\
		\hline
		$1$ & $0.21525$ \\ \hline
		$2$ & $0.33333$ \\ \hline
		$3$ & $0.37893$ \\ \hline
		$4$ & $0.39871$ \\
		\hline
	\end{tabular}
	\vspace*{0.2cm}
	\caption{The roots $R_{16}^{m}$ for $m=1,2,3,4.$} \label{table7}
\end{table}
%%%%%%%%%%%%%%%%%%%%%%%%%%%%%%%%%%%%%%%%%%%%%%%%%%%%%%%%%%%%%%%

\bigskip
\noindent
{\bf Acknowledgement.} 
The authors would like to thank Professor S. Ponnusamy for many fruitful discussions on this topic. The work of the first author is supported by CSIR, New Delhi (Grant No: 09/1022(0034)/2017-EMR-I).

\medskip
\noindent
{\bf Conflict of Interests.} The authors declare that there is no conflict of interests 
regarding the publication of this paper.

\end{document}